\documentclass{amsart}
\usepackage{amssymb,mathrsfs}

\theoremstyle{plain}
\newtheorem{theorem}{Theorem}[section]
\newtheorem{proposition}[theorem]{Proposition}
\newtheorem{lemma}[theorem]{Lemma}
\newtheorem{corollary}[theorem]{Corollary}

\theoremstyle{remark}
\newtheorem{definition}[theorem]{Dfinition}
\newtheorem{remark}[theorem]{Remark}
\newtheorem{example}[theorem]{Example}
\newtheorem{question}[theorem]{Question}

\newcommand{\AH}{\mathcal{A(H)}}
\newcommand{\B}{\mathcal{B}}
\newcommand{\bC}{\mathbb{C}}
\newcommand{\BH}{\mathcal{B(H)}}
\newcommand{\bN}{\mathbb{N}}
\newcommand{\bZ}{\mathbb{Z}}
\newcommand{\cA}{\mathcal{A}}
\newcommand{\card}{\textup{card}}
\newcommand{\cK}{{\mathcal{K}}}
\newcommand{\cH}{{\mathcal{H}}}
\newcommand{\cso}{\textsl{CSO}}
\newcommand{\dist}{\textup{dist}}
\newcommand{\eps}{\varepsilon}
\newcommand{\icso}{\textsl{ICSO}}
\newcommand{\ind}{\textup{ind}}
\newcommand{\KH}{\mathcal{K(H)}}
\newcommand{\la}{\langle}
\newcommand{\nul}{\textup{nul}~}
\newcommand{\onb}{\textsc{onb}}
\newcommand{\ra}{\rangle}
\newcommand{\ran}{\textup{ran}}
\newcommand{\rank}{\textup{rank}}
\newcommand{\rcso}{\textsl{RCSO}}

\begin{document}

\title[Reducible and irreducible approximation of CSOs]
{Reducible and irreducible approximation of complex symmetric operators}

 \author[T. Liu]{Ting Liu}
 \address{Institute of Mathematics\\
   Jilin University\\
   Changchun 130012\\
   P. R. China}
   \email{tingliu17@mails.jlu.edu.cn}

 \author[J. Zhao]{Jiayin Zhao} 
 \address{Institute of Mathematics\\Jilin University\\Changchun 130012\\P. R. China}
    \email{zhaojiayin2014@163.com}

\author[S. Zhu]{Sen Zhu}
\address{Department of Mathematics\\Jilin University\\Changchun 130012\\P. R. China}
\email{zhusen@jlu.edu.cn}

\thanks{Supported by National Natural Science Foundation of China (11671167).}



\begin{abstract}
This paper aims to study reducible and irreducible approximation in the set $\textsl{CSO}$ of all complex symmetric operators on a separable, complex Hilbert space $\mathcal H$. When ${\rm dim} \mathcal H=\infty$, it is proved that both those reducible ones and those irreducible ones are norm dense in $\textsl{CSO}$. When ${\rm dim} \mathcal H<\infty$, irreducible complex symmetric operators constitute an open, dense subset of $\textsl{CSO}$.
\end{abstract}

\maketitle


\section{Introduction}

Throughout this paper, we let $\cH$ denote a separable complex
Hilbert space endowed with the inner product $\la\cdot,\cdot\ra$. We always denote by
$\BH$ the collection of bounded linear operators on $\cH$.
Recall that an operator $T\in\BH$ is {\it irreducible} if it does not commute with any nontrivial projection; otherwise, $T$ is called {\it reducible}.

In 1968, Halmos \cite{Halmos68} proved that the set of irreducible operators is a dense $G_\delta$ in $\BH$.
Later on, Halmos \cite{Halmos70} raised ten problems in Hilbert spaces and his Problem 8 asked: \textit{Is every operator the norm limit of reducible ones?}
In order to answer the problem above, Voiculescu \cite{Vocu76} obtained the well-known noncommutative Weyl-von Neumann Theorem, which gives an affirmative answer to Halmos' question. By Voiculescu's Theorem, each operator acting on infinite-dimensional Hilbert space is approximately unitarily equivalent to a reducible operator and hence a norm limit of reducible operators. Recall that two operators $A,B\in\BH$ are {\it approximately unitarily equivalent} if there is a sequence $\{U_n\}_{n=1}^\infty$ of unitary operators such
that $U_nA-BU_n\rightarrow 0$ as $n\rightarrow\infty$.

The aim of the present study is to prove similar approximation results in the context of complex symmetric operators. To proceed we first introduce some notations and terminology.

Let $C$ be a {\it conjugation} on $\cH$, that is, $C$ is conjugate-linear, invertible with $C^{-1}=C$ and $\langle Cx, Cy\rangle=\langle y,x\rangle$ for all $x, y\in\cH$. An operator $T\in\BH$ is called $C$-{\it symmetric} if $CTC=T^*$. If $T$ is $C$-symmetric for some conjugation $C$, then $T$ is called {\it complex symmetric}. Complex symmetric operators are natural generalizations of symmetric matrices in the Hilbert space setting.  The general study of complex symmetric operators, especially in the infinite-dimensional case, was initiated by Garcia, Putinar and Wogen (see \cite{Gar06,Gar07,Gar09,Gar10} for references), and has recently received much attention (e.g., \cite{Berco,Noor15,Jung,NoorJFA,WangYao,Zhu2015OAM,Zhu2017OAM}).

Recently there has been some interest in the approximation of complex symmetric operators, mainly in infinite-dimensional case. The study on approximation issues helps people get a better understanding of the internal structure of complex symmetric operators.
Now we review some results known. For convenience, throughout the following we denote by $\cso$ the set of all complex symmetric operators on $\cH$.
As a subset of $\BH$, $\cso$ admits no linear structure or algebraic structure, although it is closed under the scalar multiplication and the adjoint operation. When $\dim\cH=\infty$, $\cso$ is not closed under both the strong operator topology and the weak operator topology; in fact, $\cso$ is dense in $\BH$ under these two topologies (see \cite[Theorem 3]{GarPooreJFA}).

In \cite{Gar09}, Garcia and Wogen raised the norm closure problem for complex symmetric operators which asked whether or not $\cso$ is norm closed. The author, Li and Ji \cite{ZhuLiJi} gave a negative answer to the above question by proving that the Kakutani shift lies in $\overline{\cso}\setminus\cso$. Almost immediately, using the unilateral shift, Garcia and Poore \cite{GarPoore13} constructed a completely different counterexample. Subsequently Garcia and Poore \cite{GarPooreJFA} constructed a large class of weighted shifts belonging to $\overline{\cso}\setminus\cso$. In \cite{GuoJiZhu}, Guo, Ji and the author provided a $C^*$-algebra approach to the study of complex symmetric operators and gave concrete characterizations for some special classes of operators to belong to $\overline{\cso}$, such as weighted shifts and essentially normal operators. Thereafter the author \cite{ZhuMA} gave a characterization of $\overline{\cso}$; in particular, it was proved that an operator $T$ lies in $\overline{\cso}$ if and only if $T$ is approximately unitarily equivalent to a complex symmetric operator. By a consequence of Voiculescu's Theorem, this implies that $\overline{\cso}\subset\cso+\KH$, where $\KH$ is the set of all compact operators on $\cH$.
These results suggest a rich structure of $\cso$.

We remark that the study on the approximation of complex symmetric operators has greatly promoted the study of complex symmetric operators and their relatives. Indeed some important progresses in the last several years are linked to and rely on results concerning the approximation of complex symmetric operators. For example, the classification of complex symmetric weighted shifts (\cite{ZhuLi}) relies on a class of fine-rank operators in $\cso$ which were constructed to show that the Kakutani shift is a norm limit of complex symmetric operators (\cite{ZhuLiJi}).
Furthermore, this inspired a decomposition theorem of complex symmetric operators (\cite{GuoZhu}).
When describing which von Neumann algebras and $C^*$-algebras can be singly generated by complex symmetric operators (\cite{ShenZhu,ZhuZhao}), many techniques are employed from approximation theory. Approximation techniques sometimes can be used to solve problems which seems unrelated to approximation. By \cite[Theorem 7.3]{GuoJiZhu}, an essentially normal operator lies in $\overline{\cso}$ precisely when it lies in $\cso$. This result is used to describe when an operator $T$ satisfies that every operator similar to $T$ is complex symmetric. In fact, such $T$ must be an algebraic operator of degree at most $2$ (see \cite[Theorem 1.2]{ZhuZhaoAFA}).

Inspired by the preceding results, the present paper aims to study reducible and irreducible approximation among the class $\cso$.
In view of approximation results on irreducible operators \cite{Halmos68} and reducible operators \cite{Vocu76}, it is natural to ask

\begin{question}\label{Q:main}
Is every complex symmetric operator a norm limit of reducible complex symmetric operators or irreducible complex symmetric operators?
\end{question}

This paper gives a complete answer to the question above. To state our main result, we give several notations.  We write $\rcso$ for the set of reducible ones in $\cso$, and  $\icso$ for the set of irreducible ones in $\cso$. Given a subset $\mathcal{E}$ of $\BH$, we denote by $\overline{\mathcal{E}}$ the norm closure of $\mathcal{E}$.

The main result of this paper is the following theorem.

\begin{theorem}\label{T:main}
\begin{enumerate}
\item[(a)] If $\dim\cH<\infty$, then $\overline{\icso}=\cso$ and $\rcso$ is a nowhere dense closed subset of $\cso$.
\item[(b)] If $\dim\cH=\infty$, then $\overline{\rcso}=\overline{\icso}=\overline{\cso}$.
\end{enumerate}
\end{theorem}

In the case that $\dim\cH=\infty$, we obtain an analogue of Voiculescu's result in the setting of complex symmetric operators.

\begin{theorem}\label{T:AueReduc}
If $T\in\BH$ is complex symmetric, then $T$ is approximately unitarily equivalent to a reducible complex symmetric operator.
\end{theorem}

A new concept called ``essentially g-normal" will play a key role in the proof of Theorem \ref{T:AueReduc}. Also the proof employs Voiculescu's noncommutative Weyl-von Neumann Theorem. The proof of Theorem \ref{T:AueReduc} provides a new approach to attack norm-approximation problems related to complex symmetric operators. In fact, the proof of Theorem \ref{T:main} (ii) relies on a corollary of Theorem \ref{T:AueReduc}.  Also,  as an application of Theorem \ref{T:AueReduc}, we shall show in Subsection \ref{S:FinitSpec} that those ones with their spectra consisting of finite components are norm dense in $\cso$ (see Theorem \ref{T:FiniteSpectra}). This is analogous to an approximation result on $\BH$ which states that those operators with their spectra consisting of finite components are norm dense in $\BH$ (\cite{Apostol}).

Given a subset $\mathcal{E}$ of $\BH$, we denote by $\overline{\mathcal{E}}^c$
the set of all operators $A\in\BH$ satisfying: for any $\eps> 0$, there exists $K\in\KH$ with $\|K\|<\eps$ such that $A+K\in \mathcal{E}$.
We call $\overline{\mathcal{E}}^c$ the {\it compact closure} of $\mathcal{E}$. It is clear that $\mathcal{E} \subset \overline{\mathcal{E}}^c\subset(\mathcal{E}+\KH)\cap\overline{\mathcal{E}}$.

By a consequence of Voiculescu's Theorem, if $A,B\in\BH$ and $A\cong_a B$, then there is a sequence $\{U_n\}_{n=1}^\infty$ of unitary operators such
that $U_nA-BU_n\in\KH$ for all $n$ and $U_nA-BU_n\rightarrow 0$ as $n\rightarrow\infty$. By Theorem \ref{T:AueReduc} and \cite[Theorem 3]{ZhuMA}, this shows that
$$\overline{\cso}^c=\overline{\cso}=\overline{\rcso}=\overline{\rcso}^c\subset\rcso+\KH.$$

By a result of Radjavi and Rosenthal \cite{Radjavi}, every operator is a small compact perturbation
of irreducible operators. In view of Theorem \ref{T:main} (b), it is natural to ask: \textit{Is every complex symmetric operator a compact perturbation or a small compact perturbation of irreducible ones?} More precisely,

\begin{question}\label{Q:smallCmpt}
Does $\cso\subset\overline{\icso}^c$ or $\cso\subset \icso+\KH$ hold?
\end{question}

It was proved in \cite[Proposition 2.4]{ZhuZhao} that all normal operators belong to $\overline{\icso}^c$.
In this paper we shall provide some other evidence for a positive answer to Question \ref{Q:smallCmpt}; see Theorem
\ref{T:smaoCmpt} and Subsection \ref{S:smallCmptPert}.

One may expect an irreducible version of Theorem \ref{T:AueReduc}. The following example excludes the possibility.

\begin{example}
Let $S$ denote the unilateral shift on $l^2(\bN)$. Denote $T=S\oplus S^*$. In view of \cite[Theorem 4.1]{ZhuLi} or \cite[Theorem 2.8]{GuoJiZhu}, $T$ is complex symmetric.
We claim that $T$ is not approximately unitarily equivalent to any irreducible complex symmetric operator.
In fact, if $A$ is irreducible, complex symmetric and $T\cong_a A$, then $A$ is essentially normal and, by \cite[Proposition 4.27]{herr89}, $T\cong A$, where $\cong$ denotes unitary equivalence. This is absurd, since $A$ is irreducible and $T$ is reducible.
\end{example}

The proof of Theorem \ref{T:main} (a) will be given in Section \ref{S:finit}. The proofs of Theorem \ref{T:main} (b) and  Theorem \ref{T:AueReduc} will be given in Section \ref{S:infi}.


\section{Finite-dimensional case}\label{S:finit}

In this section we consider Question \ref{Q:main} in the finite-dimensional case. The main result of this section is the following theorem.

\begin{theorem}\label{T:smaoCmpt}
Let $T\in\cso$ act on a separable, complex Hilbert space $\cH$. If either $T+T^*$ or $T-T^*$ is a diagonal operator, then, given $\eps>0$, there exists $K\in\KH$ with $\|K\|<\eps$ such that $T+K\in\icso$.
\end{theorem}

The proof of Theorem \ref{T:main} (a) is an immediate corollary of the preceding result.

\begin{proof}[of Theorem \ref{T:main} (a)]
Since $\dim\cH<\infty$, it is easy to verify that $\cso$ and the set of reducible operators on $\cH$ are two closed subsets of $\BH$.
Thus $\overline{\rcso}=\rcso$ and $\icso$ is an open subset of $\cso$.

On the other hand, by Theorem \ref{T:smaoCmpt}, $\cso\subset\overline{\icso}\subset\overline{\cso}=\cso$, that is, $\cso=\overline{\icso}$. Thus $\rcso$ is nowhere dense in $\cso$.
\end{proof}

As usual, if $T\in\BH$, we denote by $\ker T$ the kernel of $T$, and by $\ran~ T$ the range of $T$. The spectrum of $T$ is denoted as $\sigma(T)$.
Given a nonempty set $E$ of $\cH$, we let $\vee E$ denote the closed linear span of $E$. For $e,f\in\cH$, let $e\otimes f$ denote the rank-one operator $(e\otimes f)(x)=\la x,f\ra e$ for $x\in\cH$.

Now we are going to give the proof of Theorem \ref{T:smaoCmpt}.

\begin{proof}[of Theorem \ref{T:smaoCmpt}]
We only prove the result in the case that $\dim\cH=\infty$. The proof for the finite-dimensional case is similar.

Assume that $T=A+\textup{i}B$, where $A,B\in\BH$ are self-adjoint. Without loss of generality we assume that $A$ is diagonal and
$\sigma(A)=\{\lambda_i: i=1,2,3,\cdots\}$.
For each $i\geq 1$, denote $\cH_i=\ker(A-\lambda_i)$. Then each $\cH_i$ reduces $A$ and $\cH=\oplus_{i\geq 1}\cH_i$.

Since $T$ is complex symmetric, we may assume that $C$ is a conjugation on $\cH$ and $CTC=T^*$. One can easily verify that $CAC=A$ and $CBC=B$; in addition, we have $C(\cH_i)=\cH_i$ for $i\geq 1$. Hence each $\cH_i$ reduces $C$. For each $i$, denote $C_i=C|_{\cH_i}$. So $C_i$ is a conjugation on $\cH_i, i\geq 1$.

By \cite[Lemma 2.11]{Gar14}, we can choose an orthonormal basis (\onb, for short) $\{e_{i,j}\}_{j\in\Lambda_i}$ of $\cH_i$ such that $C_ie_{i,j}=e_{i,j}$ for all $j\in\Lambda_i$.
For $i_1,i_2\geq 1$, $j_1\in\Lambda_{i_1}$ and $j_2\in\Lambda_{i_2}$, denote $\mu^{(i_1,j_1)}_{i_2,j_2}=\la Be_{i_1,j_1}, e_{i_2,j_2}\ra$.
Since $CBC=B$, it follows that
\begin{align*}
\mu^{(i_1,j_1)}_{i_2,j_2}&=\la Be_{i_1,j_1}, e_{i_2,j_2}\ra= \la e_{i_1,j_1}, Be_{i_2,j_2}\ra\\
&=\la e_{i_1,j_1}, CBCe_{i_2,j_2}\ra=\la BCe_{i_2,j_2},Ce_{i_1,j_1}\ra\\
&=\la Be_{i_2,j_2},e_{i_1,j_1}\ra=\mu^{(i_2,j_2)}_{i_1,j_1}.
\end{align*} Note that $B$ is self-adjoint. It follows that
$$\mu^{(i_2,j_2)}_{i_1,j_1}=\mu^{(i_1,j_1)}_{i_2,j_2}=\overline{\mu^{(i_2,j_2)}_{i_1,j_1}}$$ for $i_1,i_2\geq 1$, $j_1\in\Lambda_{i_1}$ and $j_2\in\Lambda_{i_2}$.
Also one can check that $Ae_{i,j}=\lambda_ie_{i,j}$ for $i\geq 1$ and $j\in\Lambda_i$.

Fix an $\eps>0$. We can choose pairwise distinct real numbers $\{a_{i,j}: i\geq 1, j\in\Lambda_i\}$ such that
\[\sup_{i\geq 1}\sup_{j\in\Lambda_i}|a_{i,j}-\lambda_i|<\eps/2.\]
On the other hand, we can also choose nonzero real numbers $\{ b^{(i_1,j_1)}_{i_2,j_2}: i_1,i_2\geq 1, j_1\in\Lambda_{i_1},j_2\in\Lambda_{i_2} \}$ such that
$b^{(i_1,j_1)}_{i_2,j_2}=b^{(i_2,j_2)}_{i_1,j_1}$ for   $i_1,i_2\geq 1$, $j_1\in\Lambda_{i_1}$ and $j_2\in\Lambda_{i_2}$ and
\[\sum_{i_1\geq 1}\sum_{i_2\geq 1}\sum_{j_1\in\Lambda_{i_1}} \sum_{j_2\in\Lambda_{i_2}}  \big|b^{(i_1,j_1)}_{i_2,j_2}-\mu^{(i_1,j_1)}_{i_2,j_2}  \big|<\eps/2. \]
It follows that the operator $K$ defined as $K=K_1+\textrm{i}K_2$, where \begin{align*}
 K_1&=\sum_{i\geq 1}\sum_{j\in\Lambda_i}(a_{i,j}-\lambda_i) e_{i,j}\otimes e_{i,j},\\
 K_2&=\sum_{i_1\geq 1}\sum_{i_2\geq 1}\sum_{j_1\in\Lambda_{i_1}} \sum_{j_2\in\Lambda_{i_2}} \big(b^{(i_1,j_1)}_{i_2,j_2}-\mu^{(i_1,j_1)}_{i_2,j_2}\big)  e_{i_2,j_2}\otimes e_{i_1,j_1},
\end{align*}  is a compact operator on $\cH$ with $\|K\|<\eps$. One can check that $CKC=K^*$. Thus $C(K+T)C=(K+T)^*$, that is, $K+T$ is complex symmetric.
So it remains to check that $T+K$ is irreducible.

Denote $A_1=A+K_1$ and $B_1=B+K_2$. Thus $A_1,B_1$ are self-adjoint and $T+K=A_1+\textrm{i}B_1$.
Assume that $P$ is a projection commuting with $T+K$. It follows that $PA_1=A_1P$ and $PB_1=B_1P$.
Note that $A_1$ is diagonal satisfying  $\dim\ker(A_1-a_{i,j})=1$ for all $i,j$ and
$$\bigvee_{i\geq 1}\bigvee_{j\in\Lambda_i}\ker(A_1-a_{i,j})=\cH.$$
We deduce that $P=\sum_{i\geq 1}\sum_{j\in\Lambda_i} r_{i,j}e_{i,j}\otimes e_{i,j}$, where each $r_{i,j}$ is either $0$ or $1$ and the series converges  in the strong operator topology. On the other hand, one can see from $PB_1=B_1P$  that
\[r_{i_2,j_2}b^{(i_1,j_1)}_{i_2,j_2}=b^{(i_1,j_1)}_{i_2,j_2}r_{i_1,j_1}\]for $i_1,i_2\geq 1$, $j_1\in\Lambda_{i_1}$ and $j_2\in\Lambda_{i_2}$.
Since all $b^{(i_1,j_1)}_{i_2,j_2}$'s are nonzero, we have $r_{i_1,j_1}=r_{i_2,j_2}$.
It follows readily that either $P=0$ or $P=I$.
This shows that $T+K$ is irreducible.
\end{proof}

\begin{remark}
The proof of Theorem \ref{T:smaoCmpt} is inspired by Radjavi and Rosenthal \cite{Radjavi}.
\end{remark}


\section{Infinite-dimensional case}\label{S:infi}

Throughout this section, we assume that $\dim\cH=\infty$.

\subsection{Reducible approximation}

The aim of this subsection is to prove Theorem \ref{T:AueReduc}. We first make some preparation. We begin with a useful concept.

\begin{definition}
An operator $T\in\BH$ is said to be {\it g-normal} if it satisfies
\[\|p(T^*,T)\|=\|\widetilde{p}(T,T^*)\|\]
for any polynomial $p(z,w)$ in two free variables $z,w$. Here $\widetilde{p}(z,w)$ is
obtained from $p(z,w)$ by conjugating each coefficient.
\end{definition}

The notion ``g-normal" was first introduced in \cite{GuoJiZhu}. Complex symmetric operators are always g-normal. In fact, if $A\in\BH$ is $C$-symmetric, then, for each polynomial $p(z,w)$ in two free variables,
it is easy to check that
\[C(p(A^*,A))C=\widetilde{p}(A,A^*).\]  Since $C$ is isometric, it follows that
$\|p(A^*,A)\|=\|\widetilde{p}(A,A^*)\|$.
This shows that each complex symmetric operator is g-normal.

We denote by $\KH$ the ideal of all compact operators acting on $\KH$, and by
$\pi:\BH\rightarrow \AH = \BH/\KH$ the canonical projection of $\BH$ onto the
(quotient) Calkin algebra. The image $\pi(T)=T+\KH$ of  $T$ in $\AH$
will also be denoted by $\hat{T}$.

%

%

An operator $A\in\BH$ is called a {\it semi-Fredholm} operator, if $\ran A$  is closed
and either $\nul A$ or $\nul A^*$ is finite, where $\nul A:=\dim\ker A$ and $\nul
A^*:=\dim\ker A^*$; in this case, $\ind A:=\nul A-\nul A^*$ is called the {\it
index} of $A$. In particular, if $-\infty<\ind A<\infty$, then $A$ is called a {\it
Fredholm operator}. The {\it Wolf spectrum $\sigma_{lre}(A)$} and the {\it essential spectrum $\sigma_{e}(A)$} of $A$ are defined respectively as
\[\sigma_{lre}(A):=\{\lambda \in \bC: A-\lambda \textup{ is not semi-Fredholm}
\} \] and
\[\sigma_{e}(A):=\{\lambda \in \bC: A-\lambda \textup{ is not Fredholm}
\} \]

Given a unital $C^*$-algebra $\cA$ and $a\in\cA$, we let $C^*(a)$ denote the
$C^*$-subalgebra of $\cA$ generated by $a$ and the identity.

Now we are going to give the proof of Theorem \ref{T:AueReduc}.

\begin{proof}[of Theorem \ref{T:AueReduc}]
Let $T\in\BH$ and $C$ be a conjugation on $\cH$ such that $CTC=T^*$.

Denote $\cA=C^*(\hat{T})$. Choose a unital, faithful $*$-representation $\varrho$ of $C^*(\hat{T})$ on $\cH_\varrho$. Denote $A=\varrho(\hat{T})$ and $B=A^{(\infty)}$.
By \cite[Proposition 4.21 (ii)]{herr89}, we have
\begin{equation}\label{Eq:app}
T\cong_a T\oplus B .
\end{equation}

{\it Claim.}  $B$  is g-normal.

Fix a polynomial $p(z,w)$ in two free variables $z,w$. Then
\begin{align*}
\|\widetilde{p}(\hat{T},\hat{T}^*)\|&=\inf\{\|\widetilde{p}(T,T^*)+K\|:K\in\KH\}\\
&=\inf\{ \|Cp(T^*,T)C+K\|: K\in\KH\}\\
&=\inf\{\|p(T^*,T)+CKC\|: K\in\KH\}\\
&=\inf\{\|p(T^*,T)+K\|: K\in\KH\}=\|p(\hat{T}^*,\hat{T})\|.
\end{align*}
The last but one equality follows from the fact that $\KH=\{CKC: K\in\KH\}$. Since $\varrho$ is faithful, it is easy to check that $\|\widetilde{p}(A,A^*)\|=\|p(A^*,A)\|$. Since $B=A^{(\infty)}$, we obtain $\|\widetilde{p}(B,B^*)\|=\|p(B^*,B)\|$.
This proves the claim.

Note that $C^*(B)$ contains no nonzero compact operator. By \cite[Theorem 2.1]{GuoJiZhu}, $B$ is approximately unitarily equivalent to a complex symmetric operator $R$. In view of (\ref{Eq:app}),  we obtain
$$T\cong_a T\oplus B \cong_a T\oplus R.$$
Obviously, $T\oplus R$ is reducible and complex symmetric. This proves the theorem.

In the remainder, we shall show that
\begin{equation}\label{Eq:spectral}\sigma(R)=\sigma_{e}(R)=\sigma_{lre}(R)=\sigma_{lre}(T).\end{equation}

Since $R$ is complex symmetric, the equality $\sigma_{e}(R)=\sigma_{lre}(R)$ is obvious.
Now assume that $\lambda\in\bC\setminus\sigma_{e}(R)$.
So $R-\lambda$ is a Fredholm operator with $\ind(R-\lambda)=0$. From $B\cong_a R$, it follows that $\ind(B-\lambda)=0$.
Noting that $B=A^{(\infty)}$, we deduce that $A-\lambda$ is a Fredholm operator. If $A-\lambda$ is not invertible, then either $\dim\ker(A-\lambda)>0$ or $\dim\ker(A-\lambda)^*>0$. Without loss of generality, we assume the former holds. Thus $\dim\ker(B-\lambda)=\infty$, contradicting that $B-\lambda$ is a Fredholm operator. So we have shown that $A-\lambda$ is invertible. Furthermore, $R-\lambda$ is invertible. So $\lambda\notin\sigma(R)$. Thus $\sigma(R)=\sigma_{e}(R)$.

Since $ R\cong_a B$, we obtain $\sigma(R)=\sigma(B)$. From $B=A^{(\infty)}$, we have $\sigma(B)=\sigma(A)$. Noting that $\varrho$ is faithful, so
 $\sigma(A)=\sigma(\hat{T})=\sigma_e(T)$. These combining the fact that $T$ is complex symmetric imply
$$\sigma(R)=\sigma(B)=\sigma(A)=\sigma_e(T)=\sigma_{lre}(T).$$
This proves (\ref{Eq:spectral}).
\end{proof}

\begin{corollary}\label{C:ApproReduc}
If $T\in\BH$ is complex symmetric, then there exists a complex symmetric operator $R$ satisfying
\begin{enumerate}
\item[(i)] $T\cong_a T\oplus R\oplus R$, and
\item[(ii)] $\sigma(R)=\sigma_{lre}(R)=\sigma_{lre}(T)$.
\end{enumerate}
\end{corollary}

The proof of Theorem \ref{T:AueReduc} motivates a new notion.

\begin{definition}
Let $\cA$ be a unital $C^*$-algebra. An element $a\in\cA$ is said to be {\it g-normal} if it satisfies
\[\|p(a^*,a)\|=\|\widetilde{p}(a,a^*)\|\]
for any polynomial $p(z,w)$ in two free variables $z,w$. An operator $T\in\BH$ is said to be {\it essentially g-normal} if $\hat{T}$ is a g-normal element of $\AH$.
\end{definition}

From the proof of Theorem \ref{T:AueReduc}, one can see the following.

\begin{corollary}
\begin{enumerate}
\item[(i)] All complex symmetric operators are essentially g-normal.
\item[(ii)] If $T\in\BH$ is essentially g-normal, then $T\cong_a T\oplus R$ with $R$ being complex symmetric and $$\sigma(R)=\sigma_{lre}(R)=\sigma_{lre}(T).$$
\end{enumerate}
\end{corollary}

It is natural to explore the relation between g-normality and essential g-normality. We conclude this subsection with two examples.

We let $S$ denote the unilateral shift on $l^2(\bN)$ defined as
\[S(\alpha_1,\alpha_2,\alpha_3,\cdots)= (0,\alpha_1,\alpha_2,\alpha_3,\cdots),\ \ \ \forall \{\alpha_i\}_{i=1}^\infty\in l^2(\bN).\]

\begin{example}
$S$ is essentially g-normal and not g-normal. Noting that $S^*S-SS^*\in\KH$, it is easy to verify that
$p(\hat{S}^*,\hat{S})^*=\widetilde{p}(\hat{S},\hat{S}^*).$ Thus $S$ is essentially g-normal.
On the other hand, note that
$$I-S^*S=0,\ \ \ I-SS^*\ne 0.$$ So $S$ is not g-normal. However, by the B-D-F Theorem (\cite{BDF}), $S$ is unitarily equivalent to a compact perturbation of
$S^{(2)}\oplus S^*$. Note that $S^{(2)}\oplus S^*$ is g-normal. Thus $S$ is a compact perturbation of g-normal operators.
\end{example}

\begin{example}
Denote $T=S^{(\infty)}\oplus S^*$.
Note that $I-T^*T$ is compact and $I-TT^*$ is not compact. Thus $T$ is not essentially g-normal.
On the other hand,  if we denote $A=S\oplus S^*$, then it is easy to see $\|p(T^*,T)\|=\|p(A^*,A)\|$ for any polynomial $p(\cdot,\cdot)$ in two free variables. Since $A$ is clearly complex symmetric and hence g-normal, we deduce that $T$ is g-normal.
\end{example}

So far, we do not know any example of essentially g-normal operator which can not be written as ``g-normal plus compact".

\begin{question}
Is every essentially g-normal operator of the form ``g-normal plus compact"? If not, which essentially g-normal operators are of the form ``g-normal plus compact"?
\end{question}


\subsection{Spectra of complex symmetric operators}\label{S:FinitSpec}

The aim of this subsection is to give an application of Theorem \ref{T:AueReduc}. In general, the spectrum of a complex symmetric operator may
have infinitely many components, since normal operators are always complex symmetric and their spectra may be any nonempty compact set. Also we can construct non-normal examples. In fact, given $A\in\BH$ and a conjugation $C$ on $\cH$, the  operator
$T=A\oplus CA^*C$ acting on $\cH\oplus\cH$ is complex symmetric with respect to the following conjugation
\[\begin{bmatrix}
0&C\\
C&0
\end{bmatrix}\begin{matrix}
  \cH\\ \cH
\end{matrix}.\] One can check that $\sigma(T)=\sigma(A)$. If $A$ is non-normal, then so is $T$.

The following result shows that those ones with their spectra consisting of finite components are norm dense in $\cso$.

\begin{theorem}\label{T:FiniteSpectra}
Given $T\in\cso$ and $\eps>0$, there exists $K\in\BH$ with $\|K\|<\eps$ such that $T+K\in\cso$ and $\sigma(T+K)$ consists of finite components.
\end{theorem}

\begin{proof}
Clearly, we may assume that $\dim\cH=\infty$. By Corollary \ref{C:ApproReduc}, there exists $R\in\cso$ with $ \sigma(R)=\sigma_{lre}(R)=\sigma_{lre}(T)$ such that $T\cong_a T\oplus R\oplus R$.
Assume that $C$ is a conjugation on $\cH$ and $CRC=R^*$. It suffices to prove the conclusion for $T\oplus R\oplus R$.

Fix an $\eps>0$ and set $\delta=3\eps/4$. Note that $\{B(\lambda,\delta)\}_{\lambda\in\sigma(R)}$ is an open cover of $\Gamma:=\{z\in\bC: \textrm{dist}(z,\sigma(R))\leq\eps/2\}$. So there exist finite points $\lambda_1,\lambda_2,\cdots,\lambda_n$ in $\sigma(R)$ such that $\Gamma\subset\cup_{i=1}^n B(\lambda_i,\delta)$. Thus
$$\sigma(R)+B(0,\eps/2)=\sigma_{lre}(R)+B(0,\eps/2)\subset\cup_{i=1}^n B(\lambda_i,\delta).$$

{\it Claim 1.} There exists $A\in\BH$ such that
\begin{enumerate}
\item[(i)] $\|A-R\|<\eps$,  and
\item[(ii)] $\sigma(A)=\sigma_{lre}(A)=\cup_{i=1}^n \overline{B(\lambda_i,\delta)}$ consists of finite components.
\end{enumerate}

By \cite[Lemma 3.2.6]{JiangWang06}, there exists $K_1\in\KH$ with $\|K_1\|<\eps/4$ such that
\[R+K_1=\begin{bmatrix}
  \lambda_1I_1 &&& E_1\\
  &\ddots&&\vdots\\
  &&\lambda_nI_n&E_n\\
  &&& S
\end{bmatrix}\begin{matrix}\cH_1\\ \vdots\\ \cH_n \\  \cH_0 \end{matrix},\] where $\cH=\cH_1\oplus\cdots\oplus\cH_n\oplus\cH_0$, $\dim\cH_i=\infty$, $I_i$ is the identity operator on $\cH_i$ $(0\leq i\leq n)$, $\sigma(S)=\sigma_{lre}(S)=\sigma(R)$  and the entries not showing up are $0$.

For each $i$ with $1\leq i\leq n$, choose a normal operator $N_i\in\B(\cH_i)$ with $\sigma(N_i)=\overline{B(0,\delta)}$ and without eigenvalues. Set \[K_2=\begin{bmatrix}
  N_1 &&&  \\
  &\ddots&& \\
  &&N_n& \\
  &&& 0
\end{bmatrix}\begin{matrix}\cH_1\\ \vdots\\ \cH_n \\  \cH_0 \end{matrix}.\] Then $\|K_2\|=\delta$ and
\[R+K_1+K_2=\begin{bmatrix}
  \lambda_1+N_1 &&& E_1  \\
  &\ddots&&\vdots \\
  &&\lambda_n+N_n& E_n \\
  &&& S
\end{bmatrix}\begin{matrix}\cH_1\\ \vdots\\ \cH_n \\  \cH_0 \end{matrix}.\]
Set $A=R+K_1+K_2$. Then $\|R-A\|<\eps/4+\delta=\eps$. One can check that
$$\sigma(A)=\left(\cup_{i=1}^n\sigma(\lambda_i+N_i)\right)\cup \sigma(S)=\cup_{i=1}^n\sigma(\lambda_i+N_i)=\cup_{i=1}^n\overline{B(\lambda_i,\delta)}.$$
Clearly, $\sigma(A)=\sigma_{lre}(A)$ consists of finite components.
So $A$ satisfies all requirements and this proves the claim.

Set $W=T\oplus A\oplus CA^*C$. It is obvious that $W$ is complex symmetric and
\begin{align*}
  \|W-T\oplus R \oplus R\|&=\max\{ \|R-A\|, \|R-CA^*C\|\}\\
  &=\max\{ \|R-A\|, \|CR^*C-CA^*C\|\}\\
  &=\|R-A\|\\
  &<\eps.
\end{align*}
Now it remains to check that $\sigma(W)$ consists of finite components.

For a proof by contradiction, we assume that $\sigma(W)$ consists of infinitely many components. Thus we can choose countably many pairwise disjoint components $\{\Gamma_k: k=1,2,\cdots\}$ of $\sigma(W)$.

It is easy to see that $\sigma_{lre}(W)=\sigma_{lre}(A)=\cup_{i=1}^n\overline{B(\lambda_i,\delta)}$ and
$$ \ind(W-\lambda)=\ind(T-\lambda),\ \ \ \ \forall \lambda\in\bC\setminus\sigma_{lre}(W).$$Note that each $\Gamma_k$ is connected and closed.
Then for $i\in\{1,2,\cdots,n\}$ and $k\geq 1$ we have either $\Gamma_k\cap \overline{B(\lambda_i,\delta)}=\emptyset$ or $\overline{B(\lambda_i,\delta) } \subset \Gamma_k$. So there exits $k_0$ such that
\begin{equation}\label{Eq:lre}
  \Gamma_k\cap\sigma_{lre}(W) =\Gamma_k\cap \left(\cup_{i=1}^n\overline{B(\lambda_i,\delta)}\right)=\emptyset
\end{equation} whenever $k\geq k_0$.

For each $k\geq k_0$, choose $z_k\in \Gamma_k$. Then $\ind(z_k-W)=\ind(z_k-T)=0$. Denote by $\Omega_k$ the component of $\bC\setminus\sigma_{lre}(T)$ containing $z_k$. We claim that $\{z_k: k\geq k_0\}\subset\partial\sigma(T)$.
In fact, if $z_k\notin\partial\sigma(T)$, then $\Omega_k\subset\sigma(T)$ and $\partial\Omega_k\subset\sigma_{lre}(T)$, which implies that $\Gamma_k\cap\sigma_{lre}(T)\ne\emptyset$ and $\Gamma_k\cap\sigma_{lre}(W)\ne\emptyset$, a contradiction.
This shows that $\{z_k: k\geq k_0\}\subset\partial\sigma(T)$.

Since $\{z_k: k\geq k_0\}$ is an infinite subset of $\sigma(T)$, we may directly assume that $\{z_k\}$ converges to a point $z_0$ of $\sigma(T)$. So, by \cite[Chapter XI, Theorem 6.8]{Conway90}, $z_0\in\sigma_{lre}(T)$ and there exists $k_1\geq k_0$ such that $z_k\in\sigma_{lre}(T)+B(0,\eps/2)$ for $k\geq k_1$. This shows that $z_k\in\sigma_{lre}(W)$ for all $k\geq k_1$. Hence $\Gamma_k\cap \sigma_{lre}(W)\ne\emptyset$ for $k\geq k_1$. This contradicts (\ref{Eq:lre}) and therefore we conclude the proof.
\end{proof}

The following corollary can be seen from the preceding proof and will be useful later.

\begin{corollary}\label{L:SpectEnlarge}
Let $R\in\BH$ with $\sigma(R)=\sigma_{lre}(R)$. Then, given $\eps>0$, there exists $A\in\BH$ such that
\begin{enumerate}
\item[(i)] $\|A-R\|<\eps$,
\item[(ii)] $\sigma(R)+B(0, \eps/2)\subset\sigma(A)\subset\sigma(R)+B(0, \eps)$, and
\item[(iii)] $\sigma(A)=\sigma_{lre}(A)$ consists of finite components.
\end{enumerate}
\end{corollary}


\subsection{Irreducible approximation}
The aim of this subsection is to prove Theorem \ref{T:main} (b).
We need to make some preparation. First we introduce some terminology. The reader is referred to \cite[Chapter 1]{herr89} for more details.

Let $A\in\BH$. The set $\rho_{s-F}(A):=\bC\setminus\sigma_{lre}(A)$ is called the {\it
semi-Fredholm domain} of $A$. For $\lambda\in\rho_{s-F}(A)$, the {\it minimal index} of $A-\lambda$ is defined by
\[\min\cdot\ind(A-\lambda)=\min\{\nul (A-\lambda), \nul(A-\lambda)^*\}.\]
The function $\lambda\mapsto \min\cdot\ind(A-\lambda)$ is constant on every
component of $\rho_{s-F}(A)$ except for an at most denumerable subset
$\rho^s_{s-F}(A)$ without limits in $\rho_{s-F}(A)$. Each
$\lambda\in\rho^s_{s-F}(A)$ is called a {\it singular point} of the semi-Fredholm
domain of $A$, and the set
$\rho^r_{s-F}(A)=\rho_{s-F}(A)\setminus\rho^s_{s-F}(A)$ is the set of {\it
regular points}.

If $\lambda$ is an isolated point of $\sigma(A)$, then there exists an analytic Cauchy domain
$\Omega$ such that $\lambda\in\Omega$ and $[\sigma(A)\setminus\{\lambda\}]\cap\overline{\Omega}=\emptyset$. We
let $E(\lambda; A)$ denote the {\it Riesz idempotent} of $A$ corresponding to $\lambda$, that is,
\[E(\lambda; A)=\frac{1}{2\pi
\textrm{i}}\int_{\Gamma}(z-A)^{-1}\textup{d}z,\] where
$\Gamma=\partial\Omega$ is positively oriented with respect to
$\Omega$ in the sense of complex variable theory. If $\dim\ran E(\lambda;
A)<\infty$, then $\lambda$ is called a {\it normal eigenvalue} of
$A$. The set of all normal eigenvalues of $A$ will be denoted by $\sigma_0(A)$.

\begin{lemma}\label{L:Key}
Let $T\in\cso$.
If $\lambda_0\in\rho_{s-F}(T)\cap\sigma(T)$, then, given $\eps>0$, there exists $K\in\KH$ with $\|K\|<\eps$ such that $T+K$ is complex symmetric and $\lambda_0\notin\sigma(T+K)$.
\end{lemma}

\begin{proof}
Without loss of generality, we assume that $\lambda_0=0$. Assume that $C$ is a conjugation on $\cH$ such that $CTC=T^*$.

Obviously, $0<\dim\ker T=\dim\ker T^*<\infty$. Denote $n=\dim\ker T$. Assume that $\{e_i\}_{i=1}^n$ is an \onb~ of $\ker T$. Since $CTC=T^*$, it is easy to see that $\{Ce_i\}_{i=1}^n$ is an \onb~ of $\ker T^*$.

Set $K=\frac{\eps}{2}\sum_{i=1}^n (Ce_i)\otimes e_i$. Then it is easy to check that $K\in\KH$, $\|K\|<\eps$ and $CKC=K^*$. Thus $T+K$ is $C$-symmetric. Now it remains to check that $T+K$ is invertible. Since $T$ is a Fredholm operator and $\ind (T+K)=\ind~ T=0$, it suffices to prove that $T+K$ is injective.

Assume that $x\in\cH$ and $(T+K)x=0$. Thus
\[Tx=-\frac{\eps}{2}\sum_{i=1}^n \la x, e_i\ra(Ce_i).\]
Note that the vector on the right side belongs to $\vee\{Ce_i: 1\leq i\leq n\}=\ker T^*=(\ran T)^\bot.$ Thus
$$Tx=0=-\frac{\eps}{2}\sum_{i=1}^n \la x, e_i\ra(Ce_i).$$ This shows that $x\in\ker T$ and $\la x, e_i\ra=0$ for all $1\leq i\leq n$.
Since $\{e_i\}_{i=1}^n$ is an \onb~ of $\ker T$, it follows that $x=0$. So  $T+K$ is injective.
\end{proof}

Note that those invertible operators on $\cH$ constitute an open subset of $\BH$. Then the following corollary is an immediate consequence of Lemma \ref{L:Key}.

\begin{corollary}\label{C:Key}
Let $T\in\BH$ be complex symmetric and $\Gamma$ be a finite subset of $\rho(T)$, where $\rho(T)=\bC\setminus\sigma(T)$.
If $\lambda_0\in\rho_{s-F}(T)\cap\sigma(T)$, then, given $\eps>0$, there exists $K\in\KH$ with $\|K\|<\eps$ such that $T+K$ is complex symmetric and $\Gamma\cup\{\lambda_0\}\subset\rho(T+K)$.
\end{corollary}

If $T\in\BH$, we denote $\sigma_B(T):=\sigma(T)\setminus\sigma_0(T)$.

\begin{proposition}\label{P:SpectIncrese}
If $T\in\BH$ is complex symmetric, then, given $\eps>0$, there exists $K\in\KH$ with $\|K\|<\eps$ such that $T+K$ is complex symmetric and
$\sigma_B(T+K)\subset\sigma_{lre}(T)+B(0,\eps)$, where $B(0,\eps)=\{z\in\bC: |z|<\eps\}$.
\end{proposition}

\begin{proof}
Assume that $C$ is a conjugation on $\cH$ such that $CTC=T^*$, and $\{\Omega_i: i\in\Lambda\}$ are components of $\rho_{s-F}(T)$.

Denote $\Lambda_0=\{i\in\Lambda:\textup{ diameter }\Omega_i\geq\eps\}$.
 Obviously, $\Lambda_0$ is an at most finite set. Without loss of generality, assume that $\Lambda_0$ is not empty.
Choose an $\lambda_i\in\Omega_{i}\cap \rho_{s-F}^r(T)$ for $i\in\Lambda_0$.

{\it Claim.} $\exists K\in\KH$ with $\|K\|<\eps$ such that $T+K$ is complex symmetric and $\lambda_i\in\rho(T+K)$ for all $i\in\Lambda_0$.

For each $i\in \Lambda_0$, denote $m_i=\dim\ker(T-\lambda_i)$.
Set $\Lambda_1=\{i\in\Lambda_0: m_i>0\}$. If $\Lambda_1=\emptyset$, then set $K=0$.
Now we may assume that $\Lambda_1=\{1,2,\cdots,n\}$.

Note that $\lambda_i\in\rho(T)$ for $i\in\Lambda_0\setminus\Lambda_1$.
By Corollary \ref{C:Key}, we can find $K_1\in\KH$ with $\|K_1\|<\eps/2$ such that $T+K_1$ is complex symmetric and
$\{\lambda_i: i\in \Lambda_0\setminus\Lambda_1\}\cup\{\lambda_1\}\subset\rho(T+K_1)$.
Now applying the same argument to $T+K_1$, we can find $K_2\in\KH$ with $\|K_1\|<\eps/2^2$ such that $T+K_1+K_2$ is complex symmetric and
$\{\lambda_i: i\in \Lambda_0\setminus\Lambda_1\}\cup\{\lambda_1,\lambda_2\}\subset\rho(T+K_1+K_2)$.
After finitely many steps, we can find $K_3,\cdots, K_n\in\KH$ with $\|K_i\|<\eps/2^i$ such that $T+\sum_{i=1}^nK_i$ is complex symmetric and
$$\{\lambda_i: i\in\Lambda_0\}=\{\lambda_i: i\in \Lambda_0\setminus\Lambda_1\}\cup\{\lambda_i: 1\leq i\leq n\}\subset\rho(T+\sum_{i=1}^nK_i).$$
Set $K=\sum_{i=1}^n K_i$. Then $K$ satisfies the requirements of Claim.

Now we shall prove that $\sigma_B(T+K)\subset\sigma_{lre}(T)+B(0,\eps)$. Choose a $\lambda\in\bC\setminus[\sigma_{lre}(T)+B(0,\eps)]$.
It suffices to prove that $\lambda\notin\sigma_B(T+K)$.

Since  $\lambda\notin\sigma_{lre}(T)+B(0,\eps)$, it is obvious that $\lambda\in\rho_{s-F}(T)$ and $\dist(\lambda,\sigma_{lre}(T))\geq\eps$.
So there exists unique $i_0\in\Lambda$ such that $\lambda\in\Omega_{i_0}$. Note that $\partial\Omega_{i_0}\subset\sigma_{lre}(T)$. It follows that diameter $\Omega_{i_0}\geq\eps$, so $i_0\in\Lambda_0$.
Since $T+K-\lambda_{i_0}$ is invertible, it follows that $\min\cdot\ind(T+K-z)=0$ on $\Omega_{i_0}$ except for an at most countable set $\Gamma$ which has no limit points in $\Omega_{i_0}$. Noting that $T$ is complex symmetric (and hence bi-quasitriangular), this equals to say $\Omega_{i_0}\setminus\Gamma\subset\rho(T+K)$ and $\Gamma\subset\sigma_0(T+K)$. So we have either $\lambda\in\rho(T+K)$ or $\lambda\in\sigma_0(T+K)$, each of which implies $\lambda\notin\sigma_B(T+K)$.
\end{proof}

Cowen and Douglas \cite{CowenDoug} introduced an important class of operators related
to complex geometry now known as Cowen-Douglas operators.
Let $\Omega$ be a connected open subset of $\bC$ and $n$ be a positive integer.
An operator $A\in\BH$ is said to be a Cowen-Douglas
operator, denoted by $A\in B_n(\Omega)$, if $A$ satisfies
\begin{enumerate}
\item[(a)] $\Omega\subset\sigma(A)$,
\item[(b)] $\ran(A -z) = \cH$  for $z\in\Omega$,
\item[(c)] $\vee_{z\in\Omega}\ker(A-z)=\cH$, and
\item[(d)] $\dim\ker(A-z)= n$ for $z\in\Omega$.
\end{enumerate} If $A\in B_n(\Omega)$, then it is well known that $\vee_{k\geq 1}\ker(A-z)^k=\cH$ for all $z\in\Omega$.
It is easy to check that if $A\in B_1(\Omega)$ then $A$ is irreducible.

\begin{proposition}\label{P:Co-Do}
Let $T\in\BH$. Assume that $\sigma(T)=\sigma_{lre}(T)$ is connected. If $\lambda\in\bC$ and $\dist(\lambda,\sigma(T))=\delta>0$, then, given $\eps>0$, there exists $R\in B_1(\Omega)$ such that $\|R-T\|<2\delta+\eps$ and $\sigma(R)=\sigma(T)\cup\overline{\Omega}$, where $\Omega=B(\lambda,\delta)$.
\end{proposition}

\begin{proof}
Assume that $\lambda_0\in\sigma(T)$ and $|\lambda-\lambda_0|=\dist(\lambda,\sigma(T))$.
By \cite[Lemma 3.2.6]{JiangWang06}, there exists $K_1\in\KH$ with $\|K_1\|<\eps/2$ such that
\[T+K_1=\begin{bmatrix}
  \lambda_0I_0 & E\\
  0& A
\end{bmatrix}\] relative to some decomposition $\cH=\cH_0\oplus\cH_1$, where $\dim\cH_0=\infty=\dim\cH_1$, $I_0$ is the identity operator on $\cH_0$ and $\sigma(A)=\sigma_{lre}(A)=\sigma(T)$.

Denote by $S$ the unilateral shift of multiplicity one acting on $\cH_0$.
Set $$K_2=\begin{bmatrix}
  \delta S^*+(\lambda-\lambda_0)I_0 & 0\\
  0& 0
\end{bmatrix}
\begin{matrix}
\cH_0\\ \cH_1
\end{matrix}.$$ Then $K_2\in\BH$, $\|K_2\|\leq 2\delta$ and
 $$T+K_1+K_2=\begin{bmatrix}
  \lambda I_0+\delta S^* & E\\
  0& A
\end{bmatrix}\begin{matrix}
\cH_0\\ \cH_1
\end{matrix}.$$
One can check that
\begin{enumerate}
\item[(a)] $\sigma(T+K_1+K_2)=\sigma(T)\cup\overline{\Omega}$ is connected, and
\item[(b)] $\rho_{s-F}(T+K_1+K_2)=\rho(T+K_1+K_2)\cup\Omega$ and $\ind(T+K_1+K_2-z)=1$ for $z\in\Omega$.
\end{enumerate} Then, by \cite[Theorem 1.2]{herr87}, there exists $K_3\in\KH$ with $\|K_3\|<\eps/2$ such that
$T+K_1+K_2+K_3\in B_1(\Omega)$. Take $R=T+K_1+K_2+K_3$. Then $R$ satisfies all requirements.
\end{proof}

Now we are ready to complete the proof of Theorem \ref{T:main}.

\begin{proof}[of Theorem \ref{T:main} (b)] The equality $\overline{\cso}=\overline{\rcso}$ follows from Theorem \ref{T:AueReduc}.
It remains to prove $\overline{\cso}\subset\overline{\icso}$.

Let $W\in\overline{\cso}$. By \cite[Theorem 3]{ZhuMA}, $W$ is approximately unitarily equivalent to a complex symmetric operator. In view of Corollary \ref{C:ApproReduc}, there exist two complex symmetric operator $T$ and $R$ such that
$W\cong_a T\oplus R\oplus R$, where $\sigma(R)=\sigma_{lre}(R)=\sigma_{lre}(T)$. Up to unitary equivalence and a compact perturbation of arbitrarily small norm, we can directly assume that $W=T\oplus R\oplus R$ and $R,T\in\BH$.

Now fix $\eps>0$.

{\it Step 1.} Small compact perturbation of $T$.

By Proposition \ref{P:SpectIncrese}, there exists $D\in\KH$ with $\|D\|<\eps/8$ such that $T':=T+D$ is complex symmetric and \begin{equation}\label{Eq:BrowSpe}
\sigma_B(T')\subset\sigma_{lre}(T)+B(0, \eps/8). \end{equation}

{\it Step 2.} Small perturbation of $R$.

By Corollary \ref{L:SpectEnlarge}, there exists $A\in\BH$ such that
\begin{enumerate}
\item[(i)] $\|A-R\|<\eps/2$,
\item[(ii)] $\sigma(R)+B(0, \eps/4)\subset\sigma(A)$, and
\item[(iii)] $\sigma(A)=\sigma_{lre}(A)$ consists of finite components.
\end{enumerate}
Since $\sigma(R)=\sigma_{lre}(T)$, by (ii), we have \begin{equation}\label{Eq:SpecEnlar}
\sigma_{lre}(T)+B(0, \eps/4)\subset\sigma(A).
\end{equation}

Assume that $\Gamma_1,\cdots,\Gamma_n$ are all components of $\sigma(A)$. Then, by the Riesz Decomposition Theorem, there exists a decomposition $\cH=\oplus_{i=1}^n\cH_i$ with respect to which $A$ can be written as
\[A=\begin{bmatrix}
A_1&E_1&\cdots& \cdots&*\\
0&A_2&E_2&\cdots&*\\
0&0&\ddots&\ddots&\vdots\\
\vdots&\vdots&\ddots&A_{n-1}&E_{n-1}\\
0&0&\cdots&0&A_n
\end{bmatrix}\begin{matrix}
\cH_1\\ \cH_2 \\ \vdots \\ \cH_{n-1} \\ \cH_n
\end{matrix},\] where $A_i\in\B(\cH_i)$ and $\sigma(A_i)=\Gamma_i$, $i=1,\cdots,n$. It is trivial to see $\sigma(A_i)=\sigma_{lre}(A_i)$ for each $i$. Moreover, up to a compact perturbation of arbitrarily small norm, we can directly assume that $E_i\ne 0$ for all $1\leq i\leq n-1$.

For each $i$, choose $\lambda_i\in\bC\setminus\Gamma_i$ such that $\dist(\lambda_i,\Gamma_i)=\delta_i<\eps/8$. Since $\{\Gamma_i\}$ are pairwise disjoint, we can assume that $\{\Gamma_i\cup B(\lambda_i,\delta_i)^-\}$ are still pairwise disjoint. Denote $\Omega_i=B(\lambda_i,\delta_i)$, $i=1,\cdots,n$.
By Proposition \ref{P:Co-Do}, we can choose $B_i\in B_1(\Omega_i)$ such that $\|A_i-B_i\|<\eps/4$ and $\sigma(B_i)=\Gamma_i\cup\Omega_i^-$.

For each $i$, set $K_i=B_i-A_i$. Define $K=\oplus_{i=1}^n K_i$.
Then $\|K\|< \eps/8$ and
\begin{equation}\label{Eq:MatrixReptn}
 R+K=\begin{bmatrix}
B_1&E_1&\cdots&*\\
&B_2&\ddots&\vdots\\
&&\ddots&E_{n-1}\\
&&&B_n
\end{bmatrix}\begin{matrix}
\cH_1\\ \cH_2 \\ \vdots \\ \cH_n
\end{matrix};
\end{equation}the entries not shown are zero.

Since $\Omega_i\cap\sigma(A)=\emptyset$ and $\sigma_{lre}(T)+B(0,\eps/8)\subset\sigma(A)$, we obtain
$$\Omega_i\cap[\sigma_{lre}(T)+B(0,\eps/8)]=\emptyset.$$
In view of (\ref{Eq:BrowSpe}), we deduce that
$\sigma(T')\cap\Omega_i=\sigma_0(T')\cap\Omega_i$. Since $\sigma_0(T')$ is at most countable, we can choose $\mu_i\in\Omega_i\setminus\sigma(T')$.
So $\vee_{k\geq 1}\ker(B_i-\mu_i)^k=\cH_i$ for $1\leq i\leq n$.
Since $\{\sigma(B_i)\}$ are pairwise disjoint, it follows that
\begin{equation}\label{Eq:triangular}\vee\{\ker(R+K-\mu_i)^k: 1\leq i\leq s, k\geq 1\}=\oplus_{i=1}^s\cH_i, \  \ \forall 1\leq s \leq n.\end{equation}

{\it Claim.} $R+K$ is irreducible.

Assume that $P$ is a projection of $\cH$ and $P(R+K)=(R+K)P$.
For each $s$ with $1\leq s \leq n$, it can be seen from (\ref{Eq:triangular}) that $\oplus_{i=1}^s\cH_i$ are hyperinvariant under $R+K$.
Thus $P=\oplus_{i=1}^n P_i,$ where $P_i\in\B(\cH_i)$ is a projection and $P_iB_i=B_iP_i$.
Since each $B_i$ is irreducible, we have either $P_i=0$ or the identity operator on $\cH_i$.
On the other hand, from $P(R+K)=(R+K)P$, one can see $P_iE_i=E_iP_{i+1}$ for $1\leq i\leq n-1$.
Since $E_i\ne0$ for all $i$, one can deduce that $P=I$ or $P=0$. This proves the claim.

{\it Step 3.} Construction.

Assume that $C_1,C_2$ are conjugations on $\cH$ such that $C_1RC_1=R^*$ and $C_2T'C_2=(T')^*$.

Set $$W'=\begin{bmatrix}
R+K&\eps I/8&\eps I/16\\
0&T'& \eps C_2C_1/8\\
0&0& R+C_1K^*C_1
\end{bmatrix},\ \ C=\begin{bmatrix}
0&0&C_1 \\
0&C_2& 0\\
C_1&0& 0
\end{bmatrix}.$$
Then $C$ is a conjugation on $\cH^{(3)}$, $CW'C=(W')^*$ and
\[\|W'-W\|= \left\|\begin{bmatrix}
K&\eps I/8&\eps I/16\\
0&D& \eps C_2C_1/8\\
0&0& CK^*C
\end{bmatrix}\right\|<\eps.\]
Now it remains to verify that $W'$ is irreducible.

{\it Step 4.} Verification of irreducibility.

For convenience, we assume that $$W'=\begin{bmatrix}
R+K&\eps I/8&\eps I/16\\
0&T'& \eps C_2C_1/8\\
0&0& R+C_1K^*C_1
\end{bmatrix}\begin{matrix}
\cK_1\\ \cK_2 \\ \cK_3
\end{matrix}.$$

Assume that $Q$ is a projection on $\oplus_{i=1}^3\cK_i$ such that $QW'=W'Q$.

Since $\ker(B_i-\lambda)^*=\{0\}$ for all $i$ and all $\lambda\in\bC$, one can check that $\ker(R+K-\lambda)^*=\{0\}$ for all $\lambda\in\bC$.
Noting that
$C_1(R+K-\lambda)^*C_1=(R+C_1K^*C_1-\overline{\lambda})$, we deduce that $\ker(R+C_1K^*C_1-\lambda)=\{0\}$.
 In particular, $\ker(R+C_1K^*C_1-\mu_i)^*=\{0\}$ for all $1\leq i\leq n$. On the other hand, since $\mu_i\notin\sigma(T')$ for $1\leq i\leq n$,
 we have
 \begin{align*}
 &\vee\{\ker(W'-\mu_i)^k: 1\leq i\leq n, k\geq 1\}\\
 =&\vee\{\ker(R+K-\mu_i)^k: 1\leq i\leq n, k\geq 1\}=\cK_1\oplus 0\oplus 0.
 \end{align*}
Thus $\cK_1$ is hyperinvariant under $W'$. Likewise, one can show that
\begin{align*}
 &\vee\{\ker({W'}^*-\overline{\mu_i})^k: 1\leq i\leq n, k\geq 1\}\\
 =&\vee\{\ker(R^*+C_1KC_1-\overline{\mu_i})^k: 1\leq i\leq n, k\geq 1\}\\
 =&\vee\{\ker C_1(R+K-\mu_i)^k C_1: 1\leq i\leq n, k\geq 1\}\\
  =&\vee\{\ker(R+K-\mu_i)^k C_1: 1\leq i\leq n, k\geq 1\}\\
  =&C_1\Big[\vee\{\ker(R+K-\mu_i)^k: 1\leq i\leq n, k\geq 1\}\Big]\\
 =&C_1(\cH)=0\oplus0\oplus\cK_3.
 \end{align*} Thus $\cK_3$ is hyperinvariant under $(W')^*$. So $Q$ can be written as
 $$Q=\begin{bmatrix}
Q_1&*&0\\
0&Q_2& 0\\
0&*& Q_3
\end{bmatrix}\begin{matrix}
\cK_1\\ \cK_2 \\ \cK_3
\end{matrix}.$$ Since $Q$ is self-adjoint, we deduce that $Q=\oplus_{i=1}^3 Q_i$.

From $QW'=W'Q$, one can see
\begin{enumerate}
\item[(iv)] $Q_1(R+K)=(R+K)Q_1$, $Q_3(R+C_1K^*C_1)=(R+C_1K^*C_1)Q_3$, and
\item[(v)]  $Q_1(\eps I/8)=(\eps I/8)Q_2$, $Q_1(\eps I/16)=(\eps I/16)Q_3$ .
\end{enumerate}
Statement (v) implies $Q_1=Q_2=Q_3$. On the other hand, since $R+K$ is irreducible, it follows that either $Q_1=I$ or $Q_1=0$.
Hence, either $Q$ is the identity operator on $\oplus_{i=1}^3\cK_i$ or $Q=0$.
So $W'$ is irreducible.
\end{proof}

\subsection{Small compact perturbations}\label{S:smallCmptPert}

The aim of this subsection is to provide several special classes of complex symmetric operators belonging to the compact closure of $\icso$.

Recall that an operator $T$ is said to be {\it block-diagonal} if $T$ is the direct sum of some operators acting on finite-dimensional Hilbert spaces.

\begin{lemma}\label{C:SCmptPer}
Let $T\in\cso$. If $T$ is compact or block-diagonal, then, given $\eps>0$, there exists $K\in\KH$ with $\|K\|<\eps$ such that $T+K\in\icso$.
\end{lemma}

\begin{proof}
If $T$ is compact, then $T+T^*$ is self-adjoint, compact and hence diagonal. If $T$ is block-diagonal, then so is $T+T^*$. Since $T+T^*$ is self-adjoint, it follows that it is diagonal. Then, by Theorem \ref{T:smaoCmpt}, the result follows readily.
\end{proof}

Recall that an operator $T\in\BH$ is called a {\it weighted shift} if there exist an \onb~
$\{e_i\}$ and a sequence $\{w_i\}$ of complex numbers such that $Te_i=w_ie_{i+1}$ for all $i$.
If the index $i$ runs over the positive integers, then $T$ is called a
{\it unilateral weighted shift}; while if $i$ runs over integers, then $T$ is called a
{\it bilateral weighted shift.}

\begin{proposition}\label{P:WeiShift}
Let $T\in\cso$. If $T$ is a weighted shift, then, given $\eps>0$, there exists $K\in\KH$ with $\|K\|<\eps$ such that $T+K\in\icso$.
\end{proposition}

\begin{proof}
If $T$ is a unilateral weighted shift, then, by \cite[Theorem 3.1]{ZhuLi}, $T$ is block-diagonal. In view of Lemma \ref{C:SCmptPer},
the result is clear. In the remaining, we assume that $T$ is a bilateral weighted shift with weights $\{\alpha_i\}_{i\in\bZ}$ and $Te_i=\alpha_ie_{i+1}$ for $i\in\bZ$.

The proof will be divided into three cases.

{\it Case 1.} $\card\{i\in\bZ: \lambda_i=0\}=0$.

By \cite[Theorem 4.4]{ZhuLi}, it follows that there exists $k\in\bZ$ such that $|\alpha_{k-j}|=|\alpha_j|$ for all $j\in\bZ$.
We can choose $r_1,r_2\in\bC$ with $|r_1|+|r_2|<\eps$ such that $|r_1+\alpha_{0}|=|r_2+\alpha_k|>0$ and
$|r_1+\alpha_{0}|\ne|\alpha_j|$ for all $j\in\bZ\setminus\{0,k\}$.
Denote by $T_\eps$ the bilateral weighted shift with weights $\{\beta_i\}_{i\in\bZ}$ relative to the same \onb, where
\[\beta_i=\begin{cases}
\alpha_i,& i\in\bZ\setminus\{0,k\},\\
\alpha_0+r_1,& i=0,\\
\alpha_k+r_2,& i=k.
\end{cases}\] That is, $T_\eps e_i=\beta_i e_{i+1}$ for all $i$.
 Then $T-T_\eps$ is an operator of rank not greater than $2$ and $\|T-T_\eps\|<\eps$.

Since $|\beta_{k-j}|=|\beta_j|\ne 0$ for all $j$, by \cite[Theorem 4.4]{ZhuLi}, $T_\eps$ is injective and complex symmetric.
Note that $|\beta_0|=|\beta_i|$ precisely when $i=0$ or $k$. This shows that the sequence $\{|\beta_i|\}_{i\in\bZ}$ is not periodic. Using \cite[Problem 159]{Halmos82}, one can see that $T_\eps$ is irreducible.

{\it Case 2.} $1\leq\card\{i\in\bZ: \lambda_i=0\}<\infty$.

By \cite[Theorem 4.8]{ZhuLi}, this case means that $T= A^*\oplus B\oplus A $, where $A$ is an injective unilateral weighted shift and $B$ is absent or a complex symmetric operator acting a finite-dimensional space $\cK_0$ with $\sigma(B)=\{0\}$.
If $B$ is absent, then, using a similar argument as in Case 1, one can prove the conclusion. Next we deal with the latter case.

{\it Claim.} There exists $F\in\B(\cK_0)$ with $\|F\|<\eps/2$ such that $B+F$ is invertible, irreducible and complex symmetric.

Denote by $I_0$ the identity on $\cK_0$. Then $B+\eps I_0/4$ is invertible and complex symmetric.
By Lemma \ref{C:SCmptPer}, there exists $F_0\in\B(\cK_0)$ with $\|F_0\|<\eps/4$ such that $B+\eps I_0/4+F_0$
is irreducible and complex symmetric. By the upper semi-continuity of spectrum, we may also assume that $B+\eps I_0/4+F_0$ is invertible. Set $F=F_0+\eps I_0/4$. Then $F$ satisfies all requirements.
This proves the claim.

Assume that $C_0$ is the conjugation on $\cK_0$ such that $C_0(B+F)C_0=(B+F)^*$.

Denote by $\cK_1$ the underlying space of $A$. Assume that $\{f_i\}_{i=1}^\infty$ is an \onb~ of $\cK_1$ and
$Af_i=\mu_if_{i+1}$ for $i\geq 1$. Up to unitary equivalence we may assume that $\mu_i>0$ for all $i$.
For $x\in\cK_1$ with $x=\sum_iw_if_i$, define $C_1x=\sum_i\overline{w_i}f_i$. Thus $C_1$ is a conjugation on $\cK_1$ and one can verify that $C_1AC_1=A$.

Now choose a nonzero fine-rank operator $G:\cK_0\rightarrow \cK_1$ with $\|G\|<\eps/4$ and define an operator $K$ on $\cK_1\oplus\cK_0\oplus\cK_1$ as
\[K=\begin{bmatrix}
0&G&0\\
0&F&C_0G^*C_1\\
0&0&0
\end{bmatrix}\begin{matrix}
\cK_1\\ \cK_0\\ \cK_1
\end{matrix}.\]Clearly, $K$ is of finite rank, $\|K\|<\eps$ and
\[T+K=\begin{bmatrix}
A^*&G&0\\
0&B+F&C_0G^*C_1\\
0&0&A
\end{bmatrix}\begin{matrix}
\cK_1\\ \cK_0\\ \cK_1
\end{matrix}.\]

We shall show that $T+K$ is an irreducible, complex symmetric operator.
Define a conjugate-linear operator $C$ on $\cK_1\oplus \cK_0\oplus \cK_1$ as
\[C=\begin{bmatrix}
0&0&C_1\\
0&C_0&0\\
C_1&0&0
\end{bmatrix}\begin{matrix}
\cK_1\\ \cK_0\\ \cK_1
\end{matrix}.\] It is easy to check that $C$ is a conjugation and $C(T+K)C=(T+K)^*$. So $T+K$ is complex symmetric.
On the other hand, since $B+F$ is an invertible operator acting on a finite-dimensional space and $\ker (A-z)=\{0\}$ for all $z\in\bC$, it follows that
\[\bigvee_{n\geq 1}\ker(T+K)^n=\cK_1\oplus 0\oplus 0\] and \[ \bigvee_{z\in\bC, n\geq 1}\ker(T+K-z)^n=\cK_1\oplus \cK_0\oplus 0. \]
So both $\cK_1\oplus 0\oplus 0$ and $\cK_1\oplus \cK_0\oplus 0$ are hyperinvariant subspaces of $T+K$. If $P$ is a projection commuting with $T+K$, then $P$ can be written as
\[P=\begin{bmatrix}
P_1&*&*\\
0&P_2&*\\
0&0&P_3
\end{bmatrix}\begin{matrix}
\cK_1\\ \cK_0\\ \cK_1
\end{matrix}.\] Since $P=P^*$, we have $P=P_1\oplus P_2\oplus P_3$.
From $(T+K)P=P(T+K)$ one can see that
\begin{align}\label{Eq:irr1}
  P_1A^*=A^*P_1, \ \ P_3A=AP_3, \ \ P_2(B+F)=(B+F)P_2
\end{align} and
\begin{align}\label{Eq:irr2}
  P_1G=GP_2, \ \ P_2(C_0G^*C_1)=(C_0G^*C_1)P_3.
\end{align}

Since $A,(B+F)$ are both irreducible, by (\ref{Eq:irr1}), $P_i$ is either $0$ or the identity for each $i$.
Noting that $G\ne 0$, we deduce from (\ref{Eq:irr2}) that either $P=I$ or $P=0$. This shows that $T+K$ is irreducible.

{\it Case 3.} $\card\{i\in\bZ: \lambda_i=0\}=\aleph_0$.

By \cite[Theorem 4.9]{ZhuLi}, it follows that $T$ is block-diagonal. In view of Lemma \ref{C:SCmptPer},
the result is clear.  This ends the proof.
\end{proof}

Recall that an operator $T$ is called {\it binormal} if $T$ is unitarily equivalent
to an operator of the form
\[\begin{bmatrix}
N_{1,1}& N_{1,2}\\
N_{2,1} &N_{2,2}
\end{bmatrix},\]
where the entries $N_{i,j}$ are commuting normal operators acting on a Hilbert space.
Garcia and Wogen proved that every binormal operator is complex symmetric (see  \cite[Theorem 1]{Gar10}).

\begin{proposition}\label{P:binormal}
If $T\in\BH$ is binormal, then, given $\eps>0$, there exists $K\in\KH$ with $\|K\|<\eps$ such that $T+K\in\icso$.
\end{proposition}

\begin{proof}
For convenience, we directly assume that \[T=\begin{bmatrix}
N_{1,1}& N_{1,2}\\
N_{2,1} &N_{2,2}
\end{bmatrix},\]
where the entries $N_{i,j}$ are commuting normal operators acting on a Hilbert space $\cK$.

By the Weyl-von Neumann-Berg Theorem (see \cite{Berg} or \cite[page 59]{Davi96}), for given $\eps>0$, there are compact
operators $K_{i,j}$ $(1\leq i,j\leq 2)$ with $\max_{1\leq i,j\leq 2}\|K_{i,j}\|<\eps/8$
such that $N_{i,j}+K_{i,j}$ $(1\leq i,j\leq 2)$ are simultaneously diagonalizable normal operators.
Assume that
\[N_{i,j}+K_{i,j}=\textsl{diag}\{\lambda_1^{(i,j)}, \lambda_2^{(i,j)},\lambda_3^{(i,j)},\cdots\}, \ \ 1\leq i,j\leq 2,\]
relative to an \onb~ $\{e_i\}$ of $\cK$. Then
\[T+\begin{bmatrix}
K_{1,1}& K_{1,2}\\
K_{2,1} &K_{2,2}
\end{bmatrix}=\bigoplus_{i}\begin{bmatrix}
\lambda_i^{(1,1)}& \lambda_i^{(1,2)}\\
\lambda_i^{(2,1)} &\lambda_i^{(2,2)}
\end{bmatrix} \] is the direct sum of some operators on Hilbert spaces of dimension $2$ (hence all them are binormal). Since each binormal operator is complex symmetric, so is $T+K_1$, where $$K_1=\begin{bmatrix}
K_{1,1}& K_{1,2}\\
K_{2,1} &K_{2,2}
\end{bmatrix} .$$ Note that $T+K_1$ is block-diagonal. Thus, by Lemma \ref{C:SCmptPer}, there exists compact $K_2$ with $\|K_2\|<\eps/2$ such that $T+K_1+K_2\in\icso$. Put $K=K_1+K_2$.
Then $K$ satisfies all requirements.
\end{proof}

\begin{corollary}\label{C:root}
If $T\in\BH$ and $T^2$ is normal, then, given $\eps>0$, there exists $K\in\KH$ with $\|K\|<\eps$ such that $T+K\in\icso$.
\end{corollary}

\begin{proof}
By \cite[Theorem 1]{Radjavi2}, $T$ is of  the form
\[T=N\oplus \begin{bmatrix}
A& B\\
0 &-A
\end{bmatrix},\] where $A,N$ are normal and $B$ is a positive operator that commutes with $A$.
 Denote $$R=\begin{bmatrix}
A& B\\
0 &-A
\end{bmatrix}.$$ So $R$ is binormal. From the proof of Proposition \ref{P:binormal}, one can find compact
$K_1$ with $\|K_1\|<\eps/2$ such that $R+K_1$ is both complex symmetric and block-diagonal.
Using the Weyl-von Neumann-Berg Theorem, one can find compact
$K_2$ with $\|K_2\|<\eps/2$ such that $N+K_2$ is diagonal. Thus $T+(K_1\oplus K_2)$ is block-diagonal and complex symmetric.
Using Lemma \ref{C:SCmptPer}, one can see the conclusion.
\end{proof}

\begin{proposition}\label{P:Connected}
Let $T\in\cso$ with $C^*(T)\cap\KH=\{0\}$. If $\sigma(T)$ is connected,
then, given $\eps>0$, there exists $K\in\KH$ with $\|K\|<\eps$ such that $T+K\in\icso$.
\end{proposition}

\begin{proof}Assume that $C$ is a conjugation on $\cH$ such that $CTC=T^*$. We first prove two claims.

{\it Claim 1.} $T\cong_a T\oplus T\cong_a T^{(\infty)}$.

Define \begin{align*}
\varrho: C^*(T)&\longrightarrow C^*(T\oplus T),\\
X&\longmapsto X\oplus X.
\end{align*}
Then $\varrho$ is a unital, faithful representation of $C^*(T)$.
Since $C^*(T)\cap\KH=\{0\}$, it follows that $\rank X=\rank X\oplus X=\rank\varrho(X)$ for $X\in\BH$.
Then, by \cite[Theorem II.5.8]{Davi96}, $\varrho\cong_a \textrm{id}$, where $\textrm{id}$ is the identity representation of $C^*(T)$.
It follows that $T\oplus T=\varrho(T)\cong_a \textrm{id}(T)=T$. Similarly, one can prove that $T^{(\infty)}\cong_a T$.

By Claim 1, it suffices to prove that $T\oplus T$ lies in the compact closure of the class of irreducible complex symmetric operators on $\cH\oplus\cH$.

{\it Claim 2.} $\sigma(T)=\sigma_{lre}(T)$.

It suffices to prove $\sigma(T)\subset\sigma_{lre}(T)$. Assume that $z\in\bC\setminus\sigma_{lre}(T)$. From $C(z-T)C=(z-T)^*$, we deduce that
$\dim\ker(z-T)=\dim\ker(z-T)^*$ and hence $\ind(z-T)=0$. By Claim 1, $T\cong_a T^{(\infty)}$, which implies that $z-T^{(\infty)}$ is a Fredholm operator and $\ind(z-T^{(\infty)})=0$. In particular, $\dim\ker(z-T^{(\infty)})<\infty$. This shows that $\dim\ker(z-T)=0$ and $z-T$ is invertible.
Therefore $z\notin\sigma(T)$. This proves Claim 2.

%

Since $\sigma(T)$ is connected, it follows from \cite[Theorem 2.1']{JiTams} that there exists $K_1\in\KH$ with $\|K_1\|<\eps/2$ such that
\begin{enumerate}
\item[(a)] $T+K_1$ is irreducible,
\item[(b)] $(T+K_1)^*$ has no eigenvalues, and
\item[(c)] $\ker\tau_{B,T+K_1}=\{0\}$ for any $B\in\BH$ without eigenvalues, where $\tau_{B,T+K_1}$ is the Rosenblum operator on $\BH$ defined as
\[ X\longmapsto BX-X(T+K_1).\]
\end{enumerate}

We choose an injective operator $E$ on $\cH$ satisfying that $\|E\|<\eps/2$ and $CEC=E^*$.
Set $$K=\begin{bmatrix}
  K_1&E\\
  0& CK_1^*C
\end{bmatrix}\begin{matrix}
\cH\\ \cH
\end{matrix}.$$ Then $K\in\KH$, $\|K\|<\eps$ and
$$(T\oplus T) +K=\begin{bmatrix}
  T+K_1&E\\
  0& T+CK_1^*C
\end{bmatrix}=\begin{bmatrix}
  T+K_1&E\\
  0&C (T+K_1)^*C
\end{bmatrix}.$$ It is easy to verify that
$(T\oplus T) +K$ is complex symmetric with respect to the following conjugation on $\cH\oplus\cH$
$$\begin{bmatrix}
 0&C\\
 C& 0
\end{bmatrix}.$$
So now it remains to check that $(T\oplus T) +K$ is irreducible.

Assume that $P$ is a projection on $\cH\oplus\cH$ commuting with $(T\oplus T) +K$ and
\[ P=\begin{bmatrix}
  P_{1,1}& P_{1,2}\\
  P_{2,1}&  P_{2,2}
\end{bmatrix}\begin{matrix}
\cH\\ \cH
\end{matrix}. \] Since $P$ commutes with $(T\oplus T) +K$, direct computation shows that
$$C (T+K_1)^*CP_{2,1}-P_{2,1}(T+K_1)=0.$$
Noting that $ (T+K_1)^* $ and hence $C (T+K_1)^*C$ have no eigenvalue, it follows from statement (c) that $P_{2,1}=0$.
We obtain immediately that $P_{1,2}=0$. Thus $P_{1,1}(T+K_1)=(T+K_1)P_{1,1}$ and $P_{2,2}C(T+K_1)^*C=C(T+K_1)^*CP_{2,2}$.
Since $T+K_1$, $(T+K_1)^*$ (and hence $C(T+K_1)^*C$) are irreducible, we obtain $P_{1,1},P_{2,2}\in\{0,I\}$.
Note that $P_{1,1}E=EP_{2,2}$ and $E$ is injective. Thus we have either $P_{1,1}=P_{2,2}=0$ or $P_{1,1}=P_{2,2}=I$.
Thus $P$ is either $0$ or the identity on $\cH\oplus\cH$. So $(T\oplus T) +K$ is irreducible and this completes the proof.
\end{proof}

\begin{example}
Let $T\in\cso$ with connected spectrum. Set $A=T^{(\infty)}$. That is, $A$ is the direct sum of infinite copies of $T$. Thus it is easy to check that $C^*(A)$ contains no nonzero compact operators and $\sigma(A)=\sigma(T)$ is connected. By Proposition \ref{P:Connected}, $A$ lies in the compact closure of the class of irreducible complex symmetric operators on $\cH^{(\infty)}$.
\end{example}


\end{document}